\documentclass[11pt]{article}
\usepackage{amsmath,amssymb,epsf}
\usepackage{amsmath}
\usepackage[applemac]{inputenc}
\advance \textwidth 3cm
\advance \textheight 3cm
\usepackage{amsfonts}
\usepackage{latexsym}
\newtheorem{theorem}{Theorem}
\newtheorem{lemma}{Lemma}
\newtheorem{corollary}{Corollary}
\newtheorem{definition}{D\'efinition}

\newtheorem{prop}{Property}
\newtheorem{remark}{Remark}
\newenvironment{proof}[1]{\par\noindent\underline{Proof #1}:\quad}%
{\unskip\nobreak\hfil\penalty50\hskip2em\null\nobreak\hfil%
$\Box$\parfillskip0pt\par\medskip}
\title{  Orthogonal polynomials with respect of a class of Fisher-Hartwig symbols.}

\author{ Philippe Rambour\thanks{Universit\'{e} de Paris Sud,
      B\^atiment 425; F-91405
Orsay Cedex;
tel : 01 69 15 57 28 ; fax 01 69 15 60 19
      \mbox{e-mail : philippe.rambour@math.u-psud.fr}
     }}
\date{}
\begin{document}
\maketitle
  \renewcommand{\abstractname}{Abstract}
     \begin{abstract}
     \textbf{Orthogonal polynomials with respect of a class of Fisher-Hartwig symbols.}\\
     In this paper we give an asymptotic  of the coefficients of the orthogonal polynomials on the unit circle, with  
     respect of a weight of type $\displaystyle{ f : \theta \mapsto \prod_{1\le j \le M} 
     \vert 1 - e^{i(\theta_{j}-\theta)}\vert ^{2\alpha_{j}} c}$ with 
     $\theta_{j}\in ]-\pi,\pi]$, $-\frac{1}{2} < \alpha_{j}<\frac{1}{2}$  and $c$ a sufficiently smooth function. 
     \end{abstract}
           
\textbf{Secondary} 33D45, 42C05, 42C10.
\section{Introduction}
The study of the orthogonal polynomials on the unit circle is an old and difficult problem (see \cite{BS1}, \cite{BS2} or \cite{SZEG}). 
Here we are interested in the asymptotic of the coefficient of  the orthogonals polynomial with respect of an Fisher-Hartwig symbol. A Fisher-Hartwig symbol is a function $\psi$ defined on the united circle by   
 $ \psi : e^{i\theta} \mapsto \prod_{1\le j \le M} 
    \displaystyle{ \vert e^{i(\theta-\theta_{j})} - 1\vert ^{2\alpha_{j}} e^{i\beta_{j} (\theta-\theta_{j}-\pi)} c(e^{i\theta})}$ with  
   $0<\theta,\theta_{j} <2\pi$,  $-\frac{1}{2}<\Re(\alpha_{j})$ and  for all $j, 1\le j \le M$ and where the function $c$ is assumed sufficiently smooth, continuous, non zero, and have winding number zero (see \cite{BAS1}). Here we consider the class of symbols 
   $\displaystyle{ f: e^{i\theta} \mapsto \prod_{1\le j \le M} 
     \vert e^{i(\theta-\theta_{j})} - 1\vert ^{2\alpha_{j}}  c(e^{i\theta})}$ with  
     $-\frac{1}{2}<\alpha_{j}<\frac{1}{2}$ and $c$ a regular function sufficiently smooth. It is said that 
     a function $k$ is a regular function on the united circle $\mathbb T$  when $k(\theta)>0$ for all 
     $\theta \in \mathbb T$ and $k \in L^1(\mathbb T)$.  
     In  \cite{ML3} Martinez-Finkelstein, Mac-Laughin and Saff give the asymptotic behviour of this polynomials.
     If $M=2$, $\alpha_{1}= \alpha_{2}$ and  $\theta_{1}=-\theta_{2}$, $\theta_{1}\neq 0$ we can remark that these polynomials are Gegenbauer polynomials (\cite{ {B}, {RamBeau}, {WWCG}}) (see Corollary 
\ref {GEGEN}) . The main tool to compute this is the study 
of the Toeplitz matrix with symbol $f$. Given a function $h$ in $L^1 (\mathbb T)$ we denote by 
$T_{N}(h)$ the Toeplitz matrix of order $N$ with symbol $h$ the $(N+1)\times (N+1)$ matrix defined by 
$$ \left( T_{N}(h)\right) _{i+1,j+1} = \hat h(j-i) \quad \forall i,j \quad 0\le i,j \le N $$
where $\widehat m (s)$ is the Fourier coefficient of order $s$ of the function $m$ (see, for instance \cite {Bo.4} and \cite {Bo.5}). There is a close connection between Toeplitz matrices and orthogonal polynomials on the complex unit circle. Indeed the coefficients of the orthogonal polynomial of degree $N$ with respect of $h$ are also the coefficients of the last column 
of $T_N^{-1} (h)$ except for a normalisation (see \cite{Ld}). 
Here we give an asymptotic expansion of the entries  
$\left(T_{N} \left(  f_{\alpha}\right) \right)_{k+1,1}^{-1}$ (Theorem \ref{COEF}). Using 
the symmetries of the Toeplitz matrix $T_{N}(f_{\alpha})$, we deduce from this last result an asymptotic of $\left(T_{N} (f_{\alpha}) \right)_{N-k+1,N+1}^{-1}$.\\
The proof of our main Theorem  often refers to results of \cite{RS10}. 
In this last work we have treated the case of the symbols $h_{\alpha}$defined by 
$ \theta \mapsto (1-\cos \theta)^\alpha c$ whith $-\frac{1}{2} < \alpha \leq \frac{1}{2}$ and  
the same hypothesis on  $c$ as on $c_{1}$. 
We have stated the following Theorem which is an important tool in the demonstration of Theorem \ref{COEF}.
\begin{theorem}[\cite{RS10}] \label{PREDIZERO}
If $-\frac{1}{2} <\alpha \le \frac{1}{2}$, $\alpha\neq0$
we have for $c\in A(\mathbb T ,\frac{3}{2})$ and 
$0<x<1$ 
$$ c(1) \left( T_N ( h_{\alpha})\right)^{-1} _{[Nx]+1,1} =  N^{\alpha-1} \frac{1}{\Gamma
(\alpha)}x^{\alpha-1} (1-x)^\alpha + o(N^{\alpha-1}).
$$
uniformly in $x$ for $x\in [\delta_1,\delta_2]$
with $0<\delta_1<\delta_2<1$,
\end{theorem}
with $ \mathbb T = \mathbb R / 2\pi \mathbb Z$, and the definition 
\begin{definition}
For all positive real $\tau$ we denote by 
$A(\mathbb T, \tau)$ the set 
$$A(\mathbb T, \tau)= \{h \in L^2(\mathbb T)\vert 
\sum_{s\in \mathbb Z} \vert s^\tau \hat{h}(s) \vert <\infty\}$$
\end{definition}
 This theorem has also been proved for $\alpha \in 
\mathbb N^*$ in \cite{RS04} and for $\alpha\in ]\frac{1}{2},+\infty[ 
\setminus \mathbb N^*$ in \cite{RS1111}.\\
 
\vspace{1cm}

  The  results of this paper are of interest in the study of the random matrices (see \cite {De98}, \cite{D01}) and in the analysis of time series . Indeed it is known that the $n$-th covariance matrix of a time series is a positive Toeplitz matrix.
  If $\phi$ is the symbol of this Toeplitz matrix, $\phi$ is called the spectral density of 
the time series. The time series with spectral density is the function 
$\displaystyle{f: \theta \mapsto \vert e^{i\theta}-e^{i\theta_{0}}\vert ^{2\alpha} 
 \vert e^{i\theta}-e^{-i\theta_{0}}\vert ^{2\alpha} c }$ with $\theta_{0}\in ]0,\pi[$ 
are also called  
GARMA processes. Moreover 
the time series with spectral density is the function 
$\displaystyle{f: \theta \mapsto\prod_{j=1}^k \vert e^{i\theta}-e^{i\theta_{j}}\vert ^{2\alpha_{j}} 
 \vert e^{i\theta}-e^{-i\theta_{j}}\vert ^{2\alpha_{j}} c }$ with $\theta_{0}\in ]0,\pi[$ 
are k-factors
GARMA processes \cite{{AKA3},{AKA4}}.
For more on this processes we refer the reader to 
\cite{{B},{RamBeau},{WWCG}}, 
and to \cite{{Dahlhaus}, {DOT},  {GS}, {B}, {BrDa}, {KIRTEY}, {YILUHU}} for Toeplitz matrices in times series. \\
On the other hand a random matrice is characterized by the distribution of its eigenvalues. For the case of
random unitary matrices an important case is the Dyson generalized circular unitary ensemble the density of the vector $(\theta_{1},\theta_{2}, \cdots,\theta_{N})$ of eigenvalue angles is given for a $N \times N$ matrix is (\cite{Nagao}, \cite{Nagao2}, 
\cite{Nagao3}, \cite{TW01}) 
$$ P_{N} ( \theta_{1},\theta_{2},\cdots,\theta_{N}) = \prod _{1\le j\le N} f(\theta_{j}) 
\prod_{1\le j\le k \le N} \vert e^{i\theta_{j}}-e^{i\theta_{k}}\vert ^2,$$
where $f$ is generally a regular function (see 
\cite{JOHA1}), but it also can be a Fisher-Hartwing symbol. For the Dyson generalized circular ensemble 
the correlation function is written by means of the Christofel-Darboux kernel $K_{N}$ ( see
\cite{SZEG}) associated to the orthogonal polynomials with respect of the weight $f$.\\ 
Lastly it is important to observe that Theorem \ref{COEF} provides the entries and the trace of 
the matrix $T_{N}^{-1} (f)$ with 
$f= \displaystyle{\prod_{1le j\le M} \vert 1-e^{i(\theta)-\theta_{j})}\vert ^{2\alpha_{j}} 
 c}$ (see \cite{RS10}, \cite{RS09}).\\
Now we have to precise the deep link between 
the orthogonal polynomials and the inverse of the Toeplitz matrices.\\
Let $T_n(f)$ a Toeplitz matrix with symbol $f$ and $(\Phi_{n})_{n\in \mathbb N}$ the orthogonal polynomials 
with respect to $f$ (\cite{Ld}). To have the polynomial used for the prediction theory we put 
\begin{equation} \label{predizero1} 
\Phi^*_n(z)=\sum_{k=0}^n\frac {(T_n(f))^{-1}_{k+1,N+1}}
{(T_n(f))^{-1}_{N+1,N+1}}z^{k},~\mid z\mid =1.
\end{equation}
We define the polynomial $\Phi^*_n$ (see \cite{BS1}) as  
 \begin{equation} \label{predi} 
  \Phi_n ^* (z) = z^n \bar \Phi_n (\frac{1}{z}),
\end{equation}
that implies, with the symmetry of the Toeplitz matrix
\begin{equation} \label{predizero}
\Phi^*_n(z)=\sum_{k=0}^n\frac {(T_n(f))^{-1}_{k+1,1}}
{(T_n(f))^{-1}_{1,1}}z^{k},~\mid z\mid =1.
\end{equation}
The polynomials $\tilde{\Phi_{n}}=\Phi_n ^* \sqrt{(T_n(f))^{-1}_{1,1}}$ are often called predictor polynomials. As we can see in the previous formula their coefficients are, up to a normalisation, the entries of the first
column of ${T_n (f)}^{-1}.$ \\
  
\section{Main results}
\subsection{Main notations}

In all the paper we consider the symbol defined by 
$\displaystyle{ f : \theta \mapsto \prod_{1\le j \le M} 
     \vert 1 - e^{i(\theta-\theta_{j})}\vert ^{2\alpha_{j}} c}$ where 
 $c =\Bigl\vert \frac{ P}{ Q }\Bigr\vert^2$ with $P,Q \in \mathbb R [X]$,  without zeros on the united circle, 
 $-\frac{1}{2}<\alpha_{j} <\frac{1}{2}$ and $0\le \theta_{j'}\neq \theta_{j}<2\pi.$ 
 We consider also the function
 $\displaystyle{\tilde f : \theta \mapsto \prod_{1\le j \le M} 
     \vert 1 - e^{i(\theta-\theta_{j})}\vert ^{2\alpha_{j}} }$
 We have $c=c_{1}\bar c_{1}$ with $c_{1} =\frac{P}{Q}$. Obviously 
 $c_{1}\in H^{2+}(\mathbb T)$ since $H^{2+} (\mathbb T)=\{ h \in L^2(\mathbb T)\vert u<0 \implies\hat h (u) =0 \}$. If $\chi$ is the function $\theta \mapsto e^{i\theta}$ and if
 $\chi_{j}= e^{i\theta_{j}} $ for all $j, 1\le j \le M$
we put $\displaystyle{g=\prod_{j=1}^M (1-\overline{ \chi_{j}} \chi) ^{\alpha_{j}} c_{1}}$ and 
$\displaystyle{\tilde g=\prod_{j=1}^M (1-\overline{ \chi_{j}} \chi) ^{\alpha_{j}} }$. Clearly 
$g, \tilde g \in H^{2+} (\mathbb T)$ and $f=g\overline{ g}, \tilde f = \tilde g \overline{\tilde g}$.
Then we denote by 
$ \beta_{k}$ the Fourier coefficient of order $k$ $g^{-1}$ 
and by
$ \tilde \beta_{k} $ the one of 
${\tilde g}^{-1}$. 
Without loss of generality 
we assume $\beta_{0}=1$. 
Lastly for all real $\alpha$ in $]-\frac{1}{2}, \frac{1}{2}[$ we put 
$\beta_{u}^{(\alpha)} = \widehat{(1-\chi)^{-\alpha}}$. 
\subsection {Orthogonal polynomials}
\begin{theorem} \label{COEF}
Assume that for all $j \in \{1, \cdots, M \}$ we have 
$\theta_j \in ]0, 2\pi[$, $\theta_{j}\neq \theta_{j'}$ if $j\neq j'$ and 
$-\frac{1}{2} <\alpha_{M}\le\cdots \le\alpha_{j}\le \cdots \alpha_{1}<\frac{1}{2}.$  
Let $m$, $1\le m \le M$, such that  $\alpha_{j}=\alpha_{1}$ for all $j$, $1\le j \le m$.Then 
for all integer $k$, $\frac{k}{N} \rightarrow x$, $0<x<1$,  we have the 
asymptotic 
\begin{align*}
&\left( T_{N}^{-1}\left( \prod_{1\le j \le M} 
     \vert \chi \overline{ \chi_{j}}-1\vert ^{2\alpha_{j}} c\right) 
\right)_{k+1,1}
=\\ 
=& \frac {k^{\alpha_{1}-1}}{\Gamma (\alpha_{1})} (1-\frac{k}{N})^{\alpha_{1}} 
 \sum_{j=1}^m K_{j} \overline{\chi_{j} ^k} c_{1}^{-1} (\chi_{j}) + o(k^{\alpha_{1}-1})
\end{align*}
uniformly in $k$ for $x \in [\delta_{0},\delta _{1}]$, $ 0<\delta_{0}<\delta _{1}<1,$
and with $K_{j} =\displaystyle{ \prod_{h=1}^M} ( 1- \overline{\chi_{h}} \chi_{j})^{-\alpha_{h}}$ .
\end{theorem}
Then the  following statement is an  obvious consequence of Theorems \ref{COEF}.
\begin{corollary}\label{GEGEN}
Let $\chi_{0}$ be $e^{i\theta_{0}}$ with $\theta_{0}\in ]0, + \pi[$. 
With the same hypotheses as in Theorem \ref{COEF} we have 
\begin{align*}
&\left( T_{N}^{-1}\left(\vert \chi\overline{\chi_{0}}-1Ê\vert ^{2\alpha}\vert \chi\chi_{0}-1Ê\vert ^{2\alpha} c\right)
\right)_{k+1,1}=\\
&=\frac{K_{\alpha,\theta_{0},c_1}}{\Gamma(\alpha)} \cos \left(k\theta_{0}+\omega_{\alpha,\theta_0} \right) 
 k^{\alpha-1} (1-\frac{k}{N}) ^\alpha + o(k^{\alpha-1})
\end{align*}
uniformly in $k$ for $x \in [\delta_{0},\delta _{1}]$ $ 0<\delta_{0}<\delta _{1}<1$, 
and with 
$$\omega_{\alpha,\theta_0} = \alpha\frac{\pi}{2} -\alpha\theta_{0}- \arg \left(c_{1}(\chi_{0})\right),
\quad  K_{\alpha,\theta_{0},c_1} = 2^{-\alpha+1} (\sin \theta_{0})^{-\alpha} \sqrt {c_{1}^{-1} (\chi_{0})}.$$
\end{corollary}
We can also point out the asymptotic of the coefficients of order $k$ of the predictor polynomial 
when $\frac{k}{N} \rightarrow 0$. 
\begin{corollary} \label{COEF2}
With the same hypotheses as in Theorem \ref{COEF} we have, if 
$\displaystyle{ \frac{k}{N} \rightarrow 0}$ when $N$ goes to the infinity 
$$ 
\left( T_{N}^{-1}\left( \prod_{1\le j \le M} 
     \vert \chi \overline{ \chi_{j}}-1 \vert ^{2\alpha_{j}} c\right) 
\right)_{k+1,1}
=  \beta_{k} +O(\frac{1}{N}).
$$
\end{corollary}
\section{Inversion formula}
\subsection{Definitions and notations}
Let $H^{2+}(\mathbb T)$ and 
$H^{2-}(\mathbb T)$ the two subspaces of $L^2(\mathbb T)$ defined by 
$H^{2+} (\mathbb T)=\{ h \in L^2(\mathbb T)\vert u<0 \implies \hat h (u) =0 \}$ and
$H^{2-} (\mathbb T)=\{ h \in L^2(\mathbb T)\vert u\ge0 \implies \hat h (u) =0 \}$.
We denote by $\pi_{+}$ the orthogonal projector on $H^{2+}(\mathbb T)$ 
and $\pi_{-} $ the orthogonal projector on $H^{2-}(\mathbb T)$.
 It is known (see \cite{GS}) that if $f\ge 0$ and $\ln f \in L^1(\mathbb T)$ we have 
$ f = g \bar g$ with $g\in H^{2+} (\mathbb T)$. Put $\Phi_{N}= \frac{g}{\bar g} \chi^{N+1}$. Let $H_{\Phi_{N}}$ and 
$H^*_{\Phi_{N}}$ be the two Hankel operators defined respectively on $H^{2+}$ and 
$H^{2-}$ by 
$$ H_{\Phi_{N}} \, : \quad H^{2+}(\mathbb T)\rightarrow H^{2-}(\mathbb T), 
\quad \quad H_{\Phi_{N}} (\psi )= \pi_{-}( \Phi_{N}\psi),$$
 and 
 $$H^*_{\Phi_{N}}\, :  \quad  H^{2-}(\mathbb T)\rightarrow 
H^{2+}(\mathbb T), \quad \quad H^*_{\Phi_{N}} (\psi )= \pi_{+}( \bar\Phi_{N}\psi).$$
\subsection{A generalised inversion formula}
 We have stated in \cite{RS10} for a precise class of non regular functions which contains  
 $\displaystyle{\prod _{1\le j \le M} \vert \chi\bar \chi_{j}-1\vert ^{2\alpha_{j}}c}$  the following lemma (see the appendix of \cite{RS10} for the demonstration),  
\begin{lemma}\label{INVERS}
Let $f$ be an almost everywhere positive function on the torus $\mathbb T$ with
$ \ln f$, $f$, and $\frac{1}{f}$ are in $\mathbb L^1(\mathbb T)$. Then $f=g \bar g$ with 
$g\in H^{2+}(\mathbb T)$. For all trigonometric polynomials $P$ of degree at most $N$, 
we define $G_{N,f}(P)$ by 
$$ G_{N,f}(P) = \frac{1}{g} \pi_{+} \left( \frac{P}{\bar g} \right)-
\frac{1}{g} \pi_{+} \left( \Phi_{N}\sum_{s=0}^\infty \left( H^*_{\Phi_{N}} H_{\Phi_{N}} \right)^s 
\pi_{+} \bar \Phi_{N}\pi_{+}\left( \frac{P}{\bar g}\right) \right).$$
For all $P$ we have 
\begin{itemize}
\item The serie $\displaystyle{\sum_{s=0}^\infty \left( H^*_{\Phi_{N}} H_{\Phi_{N}} \right)^s 
\pi_{+} \bar \Phi_{N}\pi_{+}\left( \frac{P}{\bar g}\right)}$ converges in $L^2(\mathbb T)$.
\item
$\det \left(T_{N}(f)\right) \neq 0$ and 
$$ \left( T_{N}(f)\right)^{-1} (P) =G_{N,f}(P).$$
\end{itemize}
\end{lemma}
An obvious corollary of Lemma \ref{INVERS} is 
\begin{corollary} \label{INVERS2}
With the hypotheses of Lemma \ref{INVERS} we have 
$$ \left( T_{N}(f)\right)^{-1} _{l+1,k+1}=
\Bigl \langle \pi_{+}\left( \frac{\chi^k}{\bar g}\right)\Big \vert \left( \frac{\chi^l}{\bar g}\right)\Bigr \rangle
- \Bigl \langle \sum_{s=0}^\infty \left( H^*_{\Phi_{N}} H_{\Phi_{N}} \right)^s 
\pi_{+} \bar \Phi_{N}\pi_{+}\left( \frac{\chi^k}{\bar g}\right) \Big \vert \bar \Phi_{N} 
\left( \frac{\chi^l}{\bar g}\right) \Bigr \rangle.
$$
\end{corollary}
Lastly if $\gamma_{u}= \widehat { \frac{g}{\overline{ g}}}(u)$ we obtain as in \cite{RS10} the formal result 
\begin{align*}
\left( H_{\Phi_{N}}^* H_{\Phi_{N}}\right)^m \pi_{+}\bar  \Phi_{N} \pi_{+}\left( \frac{\chi^k}{\bar g}\right)
&=
\sum_{u=0}^k \overline{\beta_{u,\theta_{0},c_{1}} ^{(\alpha)} }\sum _{n_{0}=0}^\infty \left(  \sum _{n_{1}=1}^\infty
\bar \gamma_{-(N+1+n_{1}+n_{0}),\alpha,\theta_{0}}\right.\\
& \sum _{n_{2}=0}^\infty \gamma_{-(N+1+n_{1}+n_{2}),\alpha,\theta_{0}} \cdots 
  \sum _{n_{2m-1}=1}^\infty \bar \gamma_{-(N+1+n_{2m-1}+n_{2m-2}),\alpha,\theta_{0}}\\
 & \left.  \sum _{n_{2m}=0}^\infty\gamma_{-(N+1+n_{2m-1}+n_{2m}),\alpha,\theta_{0}} 
  \bar \gamma_{-(u-(N+1+n_{2m}),\alpha,\theta_{0})}\right) \chi^{n_{0}}
\end{align*}
\subsection{Application to the orthogonal polynomials}
With the corollary \ref{INVERS2} and the hypothesis on $\beta _{0}$ 
the equality in the corollary \ref{INVERS2} becomes, for $l=1$,
\begin{equation} \label{STAR}
\left( T_{N}(f)\right)^{-1} _{1,k+1}= \beta _{k} - \sum_{u=0}^k \beta_{k-u}  H_{N}(u)
\end{equation}
with
\begin{align*}
H_{N}(u)
&=
\sum_{m=0}^{+\infty}\left( \sum _{n_{0}=0}^\infty\gamma_{N+1+n_{0},\alpha,\theta_{0}} \left(  \sum _{n_{1}=0}^\infty
\bar \gamma_{-(N+1+n_{1}+n_{0}),\alpha,\theta_{0}}\right.\right.\\
& \sum _{n_{2}=0}^\infty \gamma_{-(N+1+n_{1}+n_{2}),\alpha,\theta_{0}} \cdots 
  \sum _{n_{2m-1}=0}^\infty \bar \gamma_{-(N+1+n_{2m-1}+n_{2m-2}),\alpha,\theta_{0}}\\
 & \left. \left. \sum _{n_{2m}=0}^\infty\gamma_{-(N+1+n_{2m-1}+n_{2m}),\alpha,\theta_{0}} 
  \bar \gamma_{(u-(N+1+n_{2m}),\alpha,\theta_{0}}\right) \right)
\end{align*}
The remainder of the paper is devoted to the computation of the coefficients 
$\beta_{k}= \widehat{g^{-1}} (k)$, $\gamma_{k}= \widehat{\frac{g}{\bar g} }$ and $H_{N}(u)$ which appears in the inversion formula. For each step we obtain the corresponding terms for the symbol $2^\alpha(1-\cos \theta) c$ mulitiplied by a rigonometric coefficient (see \cite{RS10}). That provides the expected link with the formulas in Theorem \ref{COEF}.
\section{Demonstration of Theorem \ref{COEF}}
\subsection{Asymptotic of $\beta_{k}$}
\begin{prop} \label{PROP1}
With the hypothesis of Theorem \ref{COEF} we have, for sufficiently large 
$k$, 
$$ 
\beta_{k}
=\frac{k^{\alpha_{1}-1}} {\Gamma (\alpha_{1})} \sum_{j=1}^m K_{j}
 \overline{\chi_{j}} ^k c_{1}^{-1} (\chi_{j}) + o(k^{\tau_{1}-1})
$$
uniformly in $k$,
with 
$K_{j} = \displaystyle{ \prod_{h=1,h\neq j} ^M (1-\overline{\chi_{h}}\chi_{j})^{-\alpha_{h}}}$,
and 
$ \tau_{1}=\alpha_{1}$ if $\alpha_{1}>0$ and $\tau_{1}<=alpha_{1}-\frac{1}{2}$ else.
\end{prop}
First we have to prove the lemma 
\begin{lemma} \label{PRELI}
With the hypothesis of Theorem \ref{COEF}  we have, for a sufficiently large $k$. 
$$ 
\tilde {\beta}_{k}  
= \frac{k^{\alpha_{1}-1}} {\Gamma (\alpha_{1})} \sum_{j=1}^m K_{j}
 \overline{\chi_{j} }^k + o(k^{\alpha_{1}-1})$$
uniformly in $k$, and with $\tau_{1}$ as in Property \ref{PROP1}
\end{lemma}
\begin{remark}
In these two last statements``uniformly in $k$ '' means 
$$ \forall \epsilon>0, \exists k_{\epsilon} \in \mathbb N \quad \mathrm{such}\, \mathrm{that} : \, 
\forall k, k \ge k_{\epsilon}$$
$$\Bigl \vert  \beta_{k} 
- \frac{k^{\alpha_{1}-1}} {\Gamma (\alpha_{1})} \sum_{j=1}^m K_{j}
 \overline{\chi_{j} }^k c_{1}^{-1} (\chi_{j}) \Bigr \vert <\epsilon k^{\tau_{1}-1}$$
and 
$$\Bigl \vert \tilde \beta_{k}
- \frac{k^{\alpha_{1}-1}} {\Gamma (\alpha_{1})} \sum_{j=1}^m K_{j}
 \overline{\chi_{j}} ^k \Bigr \vert <\epsilon k^{\tau_{1}-1}.$$
\end{remark}
\begin{proof}{of Lemma \ref{PRELI}} 
 Put 
 $ g_{M}= \displaystyle{\prod_{h=1}^M} (1- \overline{\chi_{h}} \chi)^{-\alpha_{h}}$ 
 and 
 $g_{M+1}= (1-\overline{\chi_{M+1}} \chi)^{-\alpha_{M+1}}$. Assume 
$\displaystyle{ \widehat{g^{-1}}(k) = \frac{k^{\alpha_{1}-1}} {\Gamma(\alpha_{1})}\sum_{j=1}^{m} K_{j}\overline{\chi_{j}}^k}$. 
Put $k_0=k^\gamma$ and $k_{1}= k^{\gamma_{1}} $ with $0<\gamma,\gamma_{1}<1$ andfor $u>k_{0},(k-k_{1})$ we have  
\begin{equation} \label{ZYG}
 \widehat{(1-\chi)^{-\alpha_{M+1}} } (u) =\frac{ u^{\alpha_{M+1}-1} }{\Gamma (\alpha_{M+1})}+O(k^{\alpha_{M+1}-2})
\end{equation}
uniformly in $u$ (see \cite{Zyg2}).
Writting for $k\ge k_{0}, $
$ \tilde{\beta}_{k} = S_{1}+S_{2}+S_{3}$ with \\
$S_{1}= \displaystyle{ \sum_{u=0}^{k_{0}} \widehat{g_{M}^{-1} }(u) \widehat{g_{M+1}^{-1} } (k-u)}$,
$S_{2}= \displaystyle{ \sum_{u=k_{0}+1}^{k-k_{1}-1} } \widehat{g_{M}^{-1} }(u) \widehat{g_{M+1}^{-1} } (k-u)$\\
and 
$S_{3}= \displaystyle{ \sum_{u=k-k_{1}}^k} \widehat{g_{M}^{-1} }(u) \widehat{g_{M+1}^{-1} } (k-u)$.
The first sum is also 
\begin{align*}
S_{1} &=  \sum_{u=0}^{k_0}  \widehat{g_{M}^{-1} }(u) \left(\widehat{g_{M+1}^{-1} } (k-u) 
- {\overline{\chi_{M+1}}}^{k-u} \frac{ (k-u)^{\alpha_{M+1}-1}} {\Gamma (\alpha_{M+1})} \right) 
\\
&+ \sum_{u=0}^{k_0} \left(  \chi_{M+1}^{u} \frac{ (k-u)^{\alpha_{M+1}-1}} {\Gamma (\alpha_{M+1})}  \right)   {\overline{\chi_{M+1}}}^{k}
\end{align*}
We observe that 
\begin{align*} 
&\sum_{u=0}^{{k_0}}  \widehat{g_{M}^{-1} }(u) \left(\widehat{g_{M+1}^{-1} } (k-u) 
- {\overline{\chi_{M+1}}}^{k-u} \frac{ (k-u)^{\alpha_{M+1}-1}} {\Gamma (\alpha_{M+1})} \right)
\\
&= \sum _{u=0}^{k_{0}} O\left( (k-u)^{\alpha_{M+1}-2}\right)= O\left( (k-u)^{\alpha_{M+1}-1}
- k^{\alpha_{M+1}-1} \right)
\end{align*}
Since  $0 \le \alpha_{1}- \alpha_{M+1}+\frac{1}{2}$ we may assume $ \gamma< \alpha_{1}- \alpha_{M+1}+\frac{1}{2}$ and we get
\begin{equation}\label{MAJOR1}
\Bigl \vert \sum_{u=0}^{k_0}  \widehat{g_{M}^{-1} }(u) \left(\widehat{g_{M+1}^{-1} } (k-u) 
- {\overline{\chi_{M+1}}}^{k-u} \frac{ (k-u)^{\alpha_{M+1}-1}} {\Gamma (\alpha_{M+1})} \right)
\Bigr \vert = o(k^{\tau_{1}-1}).
\end{equation} 
It turns out that
\begin{align*}
S_{1} &=  
 \sum_{u=0}^{k_0}  \widehat{g_{M}^{-1} }(u) 
{\overline{\chi_{M+1}}}^{k-u}\left( \frac{ (k-u)^{\alpha_{M+1}-1}- k^{\alpha_{M+1}-1}} {\Gamma (\alpha_{M+1})} \right) \\
&+  \sum_{u=0}^{k_0}  \widehat{g_{M}^{-1} }(u) 
{\overline{\chi_{M+1}}}^{k-u} \frac{k^{\alpha_{M+1}-1}}  {\Gamma (\alpha_{M+1})} +o(k^{\tau_{1}-1})
\end{align*}
with, for $\gamma< \frac{\alpha_{1}-\alpha_{M+1}+1}{2}$ 
\begin{align*}
\Bigl \vert \sum_{u=0}^{k_0}  \widehat{g_{M}^{-1} }(u) \left(
{\overline{\chi_{M+1}}}^{k-u} \frac{ (k-u)^{\alpha_{M+1}-1}- k^{\alpha_{M+1}-1}} {\Gamma (\alpha_{M+1})} \right) \Bigr \vert &= O( k^{\alpha_{M+1}-2}) \sum_{u=0}^{k_{0}} u \\
 &= O(k^{\alpha_{M-1}+2\gamma}-2)= o(k^{\tau_{1}-1}).
\end{align*}
On the other hand 
\begin{align*}
 \sum_{u=0}^{k_0}  \widehat{g_{M}^{-1} }(u) \chi_{M+1}^{u} &=
  \sum_{u=0}^{+\infty}  \widehat{g_{M}^{-1} }(u) \chi_{M+1}^{u} \\
  & -  \sum_{u=k_0+1}^{\infty}  \widehat{g_{M}^{-1} }(u) \chi_{M+1}^{u}.
  \end{align*}
  Using the appendix we get 
  $ \displaystyle{\Bigl \vert  \sum_{u=k_0+1}^{\infty}  \widehat{g_{M}^{-1} }(u) \chi_{M+1}^{u} 
  \Bigl \vert = O (k_{0}^{\alpha_{1}-1})},$
  and $k^{\alpha_{M+1}-1} k_{0}^{\alpha_{1}-1} = o( k^{\tau_{1}-1})$ since 
  $ \alpha_{M+1}-1+ \gamma(\alpha_{1}-1)<\alpha_{1}-\frac{3}{2}$.  Hence 
  $$S_{1} = \frac{k^{\alpha_{M+1}-1} }{\Gamma (\alpha_{M}+1)} \overline{\chi_{M+1}} ^k 
  \left( \prod_{j=1}^M (1-\chi_{M+1}\overline{\chi_{j}} )^{-\alpha_{j}}\right) 
  + o(k^{\tau_{1}-1})$$
  uniformly in $k$.
  Identically we get 
  $$ S_{3} =  \frac{k^{\alpha_{1}-1} }{\Gamma (\alpha_1)} \sum_{j=0}^m {\overline{ \chi_{j}}}^k 
  \left( \prod_{h=1, h\neq j}^{M+1} (1-\chi_{j}\overline{\chi_{h}})^{-\alpha_{h}}\right) +o(k^{\tau_{1}-1}),
  $$
  uniformly in $k$.
  Finally we can remark that the appendix provides \\
   $S_{2}= O (\max ( k_{0}^{\alpha_{1}-1} k ^{\alpha_{M+1}-1} , k_{1}^{\alpha_{M+1}-1} k^{\alpha_{1}-1})
   =o(k^{\tau_{1}-1})$
    uniformly in $k$.
    We have obtain 
    \begin{enumerate}
    \item
    for $\alpha_{M+1} < \alpha_{1}$, 
    $$\beta_{k} =  \frac{k^{\alpha_{1}} }{\Gamma (\alpha_1)} \sum_{j=0}^m {\overline{ \chi_{j}}}^k 
  \left( \prod_{h=1, h \neq j}^{M+1} (1-\chi_{j}\overline{\chi_{h}})^{-\alpha_{h}}\right) +o(k^{\alpha_{1}-1}),$$
  \item
  for $\alpha_{M+1} = \alpha_{1}$
\begin{align*}
 \beta_{k} &=  \frac{k^{\alpha_{1}} }{\Gamma (\alpha_1)} \left(  \sum_{j=0}^m {\overline{ \chi_{j}}}^k 
  \left( \prod_{h=1, h\neq j }^{M+1} (1-\chi_{j}\overline{\chi_{h}})^{-\alpha_{h}}\right) \right.
  \\
  &+\left.
  {\overline{ \chi_{M+1}}}^k \left( \prod _{j=1}^M (1- \chi_{M+1} \overline{ \chi_{j}} ^{-\alpha_{j} }\right)\right)
  + o(k^{\alpha_{1}-1}).
\end{align*}
    \end{enumerate}
    that ends the proof of the lemma.
\end{proof}
To ends the proof of the property we need to obtain 
$\beta_{k}$ from 
$\tilde{\beta}_{k}$ for a sufficiently large $k$. We can remark that a similar case has been  treated in \cite{RS1111} for the function $\left((1-\chi)^\alpha c_1\right)^{-1}$. Here we develop the same idea than in this last paper.
Let $c_{m}$ the coefficient of Fourier of order $m$ of the function $c_{1}^{-1}$. The hypotheses 
on $c_{1}$ imply that $c_{1}^{-1}$ is in $A(\mathbb T , p) =\{ h \in L^2 (\mathbb T) \vert
 \sum_{u\in \mathbb Z} u^p \vert \hat h (u)\vert <\infty \}$ for all positive integer $p$ ( because 
 $c_{1}^{-1} \in C^\infty (\mathbb T)$ and for all positive integer $ \vert \widehat {h^{(p)}} \vert = \frac{1}{p} \vert \widehat {h} \vert $) . 
 We have, 
 $\displaystyle{\beta_{k} = \sum_{s=0}^k \tilde\beta_{k} c_{k-s}}.$
 For $0<\nu<1$ we can write 
 $$ \sum_{s=0}^k \tilde \beta_{s} c_{k-s} = \sum_{s=0}^{k-k^\nu}\tilde  \beta_{s} c_{k-s} + \sum_{s=k-k^\nu+1 }^k \tilde \beta_{s} c_{k-s}.$$ 
Lemma \ref{PRELI} provides, with the same notations, 
$$
\sum_{s=k-k^\nu+1 }^k \tilde \beta_{s} c_{m-s} = 
\sum_{j=0}^m K_{j}\sum_{s=k-k^\nu+1 }^k  \frac{s^{\alpha_{1}} }{\Gamma (\alpha_1)} {\overline{ \chi_{j}}}^s 
c_{k-s} +R
$$
with 
 $\vert R \vert = 
   o(m^{\tau_{1}-1}) \sum_{s=k-k^\nu+1 }^k \vert c_{k-s}\vert. 
$
Since $\sum_{s\in \mathbb Z} \vert c_{s}\vert <\infty$, we have 
$$ \sum_{s=k-k^\nu+1 }^k \tilde \beta_{s} c_{k-s} = \sum_{j=0}^k K_{j} \sum_{s=k-k^\nu+1} ^k
\frac{s^{\alpha_{1}-1}}{\Gamma(\alpha_{1}} {\overline{ \chi_{j}}}^s 
c_{k-s} + o(m^{\tau_{1}-1}).$$
We have 
\begin{equation} \label{MAJOR4}
\Bigl \vert \sum_{s=k-k^\nu}^k (s^{\alpha-1} - k^{\alpha-1})  c_{k-s}\Bigr \vert \le (1-\alpha) 
O( k^{\nu+\alpha-2}) \sum_{s=k-k^\nu+1}^m \vert c_{k-s}\vert.
\end{equation}
and the convergence of $(c_{s})$ implies 
\begin{align*}
& \sum_{s=k-k^\nu} ^k \frac{s^{\alpha_{1}-1}-k^{\alpha_{1}-1}+k^{\alpha_{1}-1}} {\Gamma (\alpha_{1})} 
{\overline{ \chi_{j}}}^s c_{k-s} 
\\ &=  \frac{k^{\alpha_{1}-1}} {\Gamma (\alpha_{1})} \sum_{s=k-k^\nu} ^k 
{\overline{ \chi_{j}}}^s c_{k-s }+ O( k^{\alpha-2+\nu})\\
&\frac{k^{\alpha_{1}-1}}{\Gamma(\alpha_{1})}{\overline\chi}^k
\left( \sum_{v=0}^\infty {\overline\chi}^v c_{v }- \sum_{v=k^\nu+1}^\infty {\overline\chi}^v c_{v }\right)
\end{align*}
For all positive integer $p$ the function $c_{1} \in A(p,\mathbb T)$). Hence one can prove first  
\begin{equation}\label{MAJOR5}
 \Bigr \vert \sum_{v= k^\nu+1}^\infty e^{+iv\theta} c_{v} \Bigl \vert \le (k^{-p\nu}) \sum_{s\in \mathbb Z} 
 \vert c_{s}\vert 
\end{equation}
and secondly
$$
\sum_{s=k-k^\nu} ^k {\overline{ \chi_{j}}}^s c_{k-s }= 
\overline{\chi_{j}}^k c_{1}^{-1} (\overline{ \chi_{j}})
+O(k^{-p\nu}).
$$
 On the other hand we have (always because $c_{1}^{-1}$ in $A(\mathbb T,p)$) 
\begin{equation}\label{NEUF}
 \Bigr \vert \sum_{s=0}^{k-k^\nu} \tilde\beta_{s} 
 c_{k-s} \Bigl \vert 
\le 
 \frac{1}{k^{p\nu }} \sum_{v\in \mathbb Z} v^p \vert c_{v}\vert  
 \max_{s\in \mathbb N} (\vert \tilde \beta_{s}\vert).
 \end{equation}
 For a good choice of $p$ and $\nu$ we obtain 
  the expected formula for $\beta_{k}.$ The uniformity is provided by Lemma \ref 
  {PRELI} and the equation (\ref{MAJOR4}), (\ref{MAJOR5}) and (\ref {NEUF}).

\subsection{Estimation of the Fourier coefficients of $\frac{g}{\overline{g}}$.}
Let $\gamma_{k}$ be $ \widehat{\frac{g} 
{\overline{g }} }(k) $ and  $\tilde\gamma_{k}$ be $ \widehat{\frac{\tilde g} 
{\overline{\tilde g }}} (k). $ 

\begin{prop} \label{prop2}
With the hypothesis of Theorem \ref{COEF} we have, for all integer $k \ge 0$  sufficiently large
$$ 
\gamma_{-k}
= \frac{1}{k}\sum_{j=1}^M  \frac { \sin (\pi\alpha_{j})} {\pi} H_{j}
 \frac{c_{1}(\chi_{j})}{\overline{c_{1}(\chi_{j})}} 
 {\overline{\chi_{j}}}^k+o(k^{\min(\alpha_{1}-1,-1)})$$

 uniformly in $k$ and with $H_{j} = \displaystyle{ \prod_{j=1,h\neq}^M \left( \frac{\overline{\chi_{h}}\chi_{j}-1}
 {\chi_{h}\overline{\chi_{j}}-1}\right)^{\alpha_{j}}}$.
\end{prop} 
First we have to prove the lemma 
\begin{lemma} \label{PRELI2}
With the hypothesis of Theorem \ref{COEF} we have, for all integer $k \ge 0$  sufficiently large
$$ 
\tilde \gamma_{-k}
= \frac{1}{k}\sum_{j=1}^M  \frac { \sin (\pi\alpha_{j})} {\pi} H_{j} {\overline{ \chi_{j}}}^k
+o(k^{\min(\alpha_{1}-1,-1)})$$
 uniformly in $k$.
 \end{lemma} 
\begin {proof}{of Lemma \ref{PRELI2}} 
In all this proof we denote respectively by $\gamma_{1,k}, \gamma_{2,k}$ the 
Fourier coefficient of order $k$ of $\displaystyle{ \prod_{j=1}^{M-1} \left( \frac{\overline{\chi_{h}}\chi-1}
 {\chi_{h}\overline \chi-1}\right)^{\alpha_{j}}}$
 and 
$\left(\frac{\chi\overline{\chi_{M}}-1)}{(\overline \chi  \chi_{M}-1)}\right)^{\alpha_{M}}$. Clearly 
$\gamma_{2,k} = (\bar\chi_{M})^k \frac{\sin \pi \alpha_{M}}{\pi} \frac{1}{k+\alpha_{M}}
=(\bar\chi_{M})^k \gamma_{3,k}.$
Assume $k \ge 0$ and $\gamma_{1,k} =\displaystyle{ \frac{1}{k}\sum_{j=1}^M  \frac { \sin (\pi\alpha_{j})} {\pi} H'_{j} {\overline{ \chi_{j}}}^k+o(\frac{1}{k}})$ with 
$H'_{j} = \displaystyle{ \prod_{j=1, h\neq j}^{M-1} \left( \frac{\overline{\chi_{h}}\chi_{j}-1}
 {\chi_{h}\overline{\chi_{j}}-1}\right)^{\alpha_{j}}}$.
Assume also $ k\ge 0$. We have
$ \displaystyle{
\gamma_{-k}= \sum_{v+u=-k} \gamma_{1,u} \gamma_{2,v}}.$
For  $k_{0}= k^\tau$, $0<\tau<1$ we can split this sum into
\begin {align*}
&\sum_{u< -k-k_{0}} \gamma_{1,u} \gamma_{2,-k-u} + \sum_{u=-k-k_{0}}^{-k+k_{0}} 
\gamma_{1,u} \gamma_{2,-k-u}
+ \sum_{u=-k+k_{0}+1}^{-k_{0}-1} \gamma_{1,u} \gamma_{2,-k-u} \\
&+\sum_{u= -k_{0}}^{k_{0}} \gamma_{1,u} \gamma_{2,k-u} 
+\sum_{u>k_{0}} \gamma_{1,u} \gamma_{2,-k-u}.
\end{align*}
Write
$$\sum_{u=-k_{0}}^{k_{0}} \gamma_{1,u} \gamma_{2,-k-u} 
= \sum_{u=-k_{0}}^{k_{0}} \gamma_{1,u} 
(\bar \chi_{M})^{k+u}  (\gamma_{3,-k-u} - \gamma_{3,-k}+ \gamma_{3,-k}).$$
Since
\begin{equation} \label{UNIF1}
\sum_{u=- k_{0}}^{k_{0}} \gamma_{1,u} (\bar \chi_{M})^{k+u}  ( \gamma_{3,-k-u} -\gamma_{3,-k})
= \frac{\sin (\pi \alpha)}{\pi}\sum_{u=- k_{0}}^{k_{0}} \gamma_{1,u} (\bar \chi_{M})^{k+u} \frac{-u}{(k+u+\alpha)(k+\alpha)}
\end{equation}
it follows that (always with the appendix)
\begin{align*}
\sum_{u=-k_{0}}^{k_{0}} \gamma_{1,u} \gamma_{2,-k-u} &=  \gamma_{3,-k}  \sum_{u=-k_{0}}^{k_{0}}
\gamma_{1,u} (\bar \chi_{M})^{-k-u} + O(k_{0} k^{-2})\\
&=   \gamma_{3,-k}  (\chi_{M})^k \sum_{\vert u\vert \ge k_{0} }  \gamma_{1u} \chi_{M}^{u}
+O(k_{0} k^{-2})\\
&= \gamma_{3,-k}  (\chi_{M})^k \prod_{j=1}^{M-1} \left( \frac{\overline{\chi_{h}}\chi_{M}-1}
 {\chi_{h}\overline {\chi_{M}}-1}\right)^{\alpha_{j}}
 +
O\left((k_{0}k)^{-1}\right) +O(k_{0} k^{-2})\\
&= \gamma_{3,-k}  (\chi_{M})^k \prod_{j=1}^{M-1} \left( \frac{\overline{\chi_{h}}\chi_{M}-1}
 {\chi_{h}\overline {\chi_{M}}-1}\right)^{\alpha_{j}}
 +
O(k^{\tau-2}).
\end{align*}
In the same way we have 
$$
\sum_{u=-k -k_{0}}^{-k+k_{0}} \gamma_{1,u} \gamma_{2,k-u} =
  \sum_{j=1}^M \frac{\sin \pi \alpha_{j}}{\pi} H'_{j} {\overline{\chi_{j}}}^k 
 \left( \frac{\overline{\chi_{M}} \chi_{j}-1} {\chi_{M}\overline{\chi_{j}} -1}\right)^{\alpha_{M}}
O(k^{\tau-2}).
$$
Now using the appendix it is easy to see that 
\begin{equation}\label{UNIF2}
 \sum_{u< -k-k_0} \gamma_{1,u} \gamma_{2,-k-u} \le M_{1} (k_{0}k)^{-1}
 \end{equation}
\begin{equation}\label{UNIF3}
\sum_{u> k_{0}} \gamma_{1,u} \gamma_{2,-k-u} \le M_{2}(k_{0}k)^{-1}
\end{equation}
with $M_{1}$ and $M_{2}$ no depending from $k$.
For the sum  $S=\displaystyle{\sum_{u=-k+k_{0}+1}^{-k_{0}-1} \gamma_{1,u} \gamma_{2,-k-u}} $
we remark, using an Abel summation, that 
$$\vert S\vert  \le M_{3} (k_{0}k)^{-1}  + \sum_{u=-k+k_{0}+1}^{-k_{0}-1} 
\Bigl \vert \frac{ 1}{(u+\alpha)(k-u+\alpha)} - \frac{1}{(u+1+\alpha)(k-u-1+\alpha)}\Bigr \vert $$
with $M_{3}$ no depending from $k$.
Consequently
\begin{equation} \label{UNIF4}
 \vert S \vert  \le M_{3}(k_{0}k)^{-1} + \sum_{u=-k+k_{0}+1}^{-k_{0}-1} \frac{k-2u} {(k-u)^2 u^2}.
 \end{equation}
Then Euler and Mac-Laurin formula provides the upper bound 
$$ \vert S \vert \le O\left( (k_{0}k)^{-1}\right) + \int_{-k+k_{0}+1}^{-k_{0}-1}\frac{k-2u} {(k-u)^2 u^2} du . $$
Since
$$ \int_{-k+k_{0}+1}^{-k_{0}-1}\frac{k-2u} {(k-u)^2 u^2} du \le \frac{3 k}{(k+k_{0})^2} 
\int_{-k+k_{0}+1}^{-k_{0}-1}\frac{1}{u^2} du$$
we get finally
$$ \sum_{u<-k+k_{0}+1}^{-k_0-1} \gamma_{1,u} \gamma_{2,k-u} = O\left((k_{0}k)^{-1}\right)$$
and
$$\tilde \gamma_{-k} 
= \frac{1}{k}\sum_{j=1}^M  \frac { \sin (\pi\alpha_{j})} {\pi} H_{j} {\overline{ \chi_{j}}}^k+ O\left( (k_{0}k)^{-1}\right) +O(k^{\alpha-2}).$$
Then with a good choice of $\tau$ we obtain the expected formula. The uniformity is a direct consequence of the equations (\ref{UNIF1}), (\ref{UNIF2}), (\ref{UNIF3}), (\ref{UNIF4}).
\end{proof}
The rest of the proof  of Lemma \ref{PRELI2} can be treated as the end of the proof of property \ref{PROP1}.

\subsection {Expression of $\left(T_{N}^{-1} \left( f\right)\right)_{k+1,1}$.}
First we have to prove the next lemma 
\begin{lemma}\label{INVERS3}
For $\alpha \in ]-\frac{1}{2}, \frac{1}{2}[$ we have a function 
$F_{N,\alpha}\in C^1[0 ,\delta ]$ for all $\delta \in ]0,1[$, satisfying the properties  
\begin{itemize}
\item [i)]
$$Ê\forall z \in [0,\delta [ \quad \vert F_{N,\alpha}(z)\vert \le K_{0}(1+\vert\ln (1-z+\frac{1+\alpha}{N})\vert )$$
where $K_{0}$ is a constant no depending from $N$.
\item [ii)]
$F_N$ and $F'_N$ have a modulus of continuity 
no depending from $N$.
\item [iii)]
with the notations of Theorem \ref{COEF} we have 
\begin{align*}
& \left(T_{N}^{-1} \left(f\right)\right)_{k+1,1}
=
\\ &=  \beta_{k} -\frac{1}{N} \sum_{u=0}^k \beta_{k-u} \left( \sum_{j=1}^M F_{N,\alpha_{j}} (\frac{u}{N})
 {\overline{\chi_{j}}}^{u}\right) +R_{N,\alpha_{1}}
\end{align*}
uniformly in $k$, $0\le k\le N$, with 
$$R_{N,\alpha_{1}} = o\left(N^{-1} \sum_{u=0}^k  \beta_{k-u} \left( \sum_{j=1}^M F_{N,\alpha_{j}} (\frac{u}{N})
 {\overline{\chi_{j}}}^{}u\right)\right) \quad \mathrm{if} \quad \alpha>0$$
and
$$R_{N,\alpha_{1}} = o\left(N^{\alpha_{\alpha_{1}}-1} \sum_{u=0}^k \ \beta_{k-u} \left( \sum_{j=1}^M 
F_{N,\alpha_{j}} (\frac{u}{N})
 {\overline{\chi_{j}}}^u\right)\right)\quad \mathrm{if} \quad \alpha<0$$
\end{itemize}
\end{lemma}
\begin{remark} \label{REMARQUE2} 

\textbf{(Proof of the corollary \ref{COEF2})}
for $\frac{k}{N}\rightarrow 0$ 
Lemma \ref{INVERS3} and the continuity of the function $F_{\alpha}$ provide 

$$  \left(T_{N}^{-1} \left(f\right)\right)_{k,1}
= \beta_{k} + \frac{1}{N} \sum_{u=0}^k \beta_{k-u} \left(\sum_{j=0}^M F_{N,\alpha_{j}} (0) {\overline{\chi_{j}}}^u \right)\left(1+o(1)\right).
$$
Since $F_{N,\alpha} (0)= \alpha^2+o(1)$ (see \cite{RS10}) the hypothesis $ \beta_{0} =1$
and the formula (\ref{STAR}) 
imply the corollary.
\end{remark}
\begin{proof}{of the lemma \ref{INVERS3}}
As for \cite{RS10} and using the inversion formula 
and Corollary  \ref{INVERS2} we have to consider the sums 
\begin{align*}
H_{p,N}(u) &= \left (  
\sum_{n_{0}=0}^\infty \gamma_{-(N+1+n_{0}) }
\sum_{n_{1}=0}^\infty \overline{\gamma_{-(N+1+n_{1}+n_{0}) }}
\sum_{n_{2}=0}^\infty \gamma_{-(N+1+n_{1}+n_{2}) }\right.
\times \cdots 
\\
&\times \left.\sum_{n_{2m-1}=0}^\infty \overline{ \gamma_{-(N+1+n_{2p-2}+n_{2p-1}) }}
\sum_{n_{2p}=0}^\infty \gamma_{-(N+1+n_{2m-1}+n_{2m}) }
\overline{\gamma_{u-(N+1+n_{2p})}}\right).
\end{align*}
If 
$$ S_{2p} = \sum_{n_{2p}=0}^\infty \gamma_{-(N+1+n_{2p-1}+n_{2p}) }
\overline{\gamma_{u-(N+1+n_{2p})}}$$
 we can write, following the previous Lemma, 
 $ S_{2p} = S_{2p,0}+ S_{2p,1} +R_{2p,\alpha_{1}}$ with 
 \begin{align*}
 S_{2p,0} &=  
  \sum_{n_{2p}=0}^\infty
\left(   \sum_{j=0}^M \left( \frac{\sin { \pi \alpha_{j}}}{\pi}\right)^2 {\overline{ \chi_{j}}}^{n_{2p-1}+u} \right.\\
& \left. \frac {1} {N+1+n_{2p-1}+n_{2p}+\alpha_{j}}\frac{1}{N+1+n_{2p}-u+\alpha_{j}}\right)
\end{align*}
  \begin{align*}
 S_{2p,1} &=  
  \sum_{n_{2p}=0}^\infty \left(   \sum_{j j'=0 j\neq j'}^M H_{j} \overline{H(j')}  \frac{\sin { \pi \alpha_{j}}}{\pi} \frac{\sin { \pi \alpha_{j'}}}{\pi}  \frac{ c_{1}(\chi^j)}{\overline{c_{1}(\chi^j)}} 
   \frac{ c_{1}(\chi^{j^\prime})}{\overline{c_{1}(\chi^{j^\prime})}}\right.\\  
& \left. {\overline{ \chi_{j}}}^{N+1+n_{2p}+n_{2p-1}} \chi_{j'}^{N+1+n_{2p}-u} \frac {1} {N+1+n_{2p-1}+n_{2p}+\alpha_{j}}\frac{1}{N+1+n_{2p}-u+\alpha_{j'}}\right)
\end{align*}
   Let us study the order of $S_{2p,1}.$
To do this we have to evaluate the order of the expression
$$ \sum_{j=0}^H \chi_{0}^{j} \frac {1} {N+1+n_{2m-1}+j+\alpha}\frac{1}{N+1+j-u+\alpha} $$
where $H$ goes to the infinity and $N=o(H)$.
As for the previous proofs it is clear that this sum is bounded by 
$$\sum_{j=0}^M \Bigl \vert \frac{1} {N+2+n_{2p-1}+j}\frac{1}{N+2+j-u} 
-\frac{1} {N+1+n_{2p-1}+j}\frac{1}{N+1+j-u}.\Bigr \vert $$
Obviously
\begin{align*}
&\Bigl \vert \frac{1} {N+2+n_{2p-1}+j}\frac{1}{N+2+j-u} 
-\frac{1} {N+1+n_{2p-1}+j}\frac{1}{N+1+j-u}\Bigr \vert  \\
&\le \Bigl \vert \frac{2N+2 +2 j+n_{2p-1}-u} { (N+1+n_{2p-1}+j)^2(N+1+j-u)^2} \Bigr \vert 
\end{align*}
and 
\begin{align*}
&\Bigl \vert \frac{2N+2+2 j+n_{2p-1}-u} { (N+1+n_{2p-1}+j)^2(N+1+j-u)^2} \Bigr \vert\\
&= \Bigl \vert \frac{1}{N+1+j +n_{2p-1}}+\frac{1}{N+1+j -u}\Bigr \vert 
\frac{1}{(N+1+j +n_{2p-1}) (N+1+j -u)} \\
&\le  \frac{1}{N} \frac{1}{(N+1+j +n_{2p-1}) (N+1+j -u)}.
\end{align*}
In the other hand we have, for $\alpha_{1} \in ]0, \frac{1}{2}[$ 
$$ R_{2p,\alpha_{1}} = o\left( \sum_{j=0}^\infty \frac{1}{N+1+n_{2p-1}+n_{2p}}
\frac{1}{N+1+n_{2p}-u}\right)$$
and for $\alpha_{1}\in]-\frac{1}{2},0[.$
$$ 
R_{2p,\alpha_{1}}= o\left(N^{\alpha_{1}} \sum_{j=0}^\infty \frac{1}{N+1+n_{2p-1}+n_{2p}}
\frac{1}{N+1+n_{2p}-u}\right).
$$

Hence we can write 
 $$ S_{2p}= S'_{2p} \left( \sum_{j=0}^M \frac{\sin \pi\alpha_{j}}{\pi} ^2 {\overline {\chi_{j}}}^{n_{2p-1}+u} +r_{m}\right),$$
 with 
 $$ S'_{2p}= \sum_{j=0}^{+ \infty} \frac{1} {N+1+n_{2m-1}+n_{2m}}\frac{1}{N+1+n_{2m}-u}.$$
 and 
 \[ \left\{ 
\begin{array}{cccc}
  r_{m,\alpha_{1}}  =  &  o(1)            &\mathrm{if} &\quad  \alpha\in]0, \frac{1}{2}[\\
   r_{m,\alpha_{1}}  =  &  o(N^{\alpha_{1}})&\mathrm{if} &\quad \alpha\in]- \frac{1}{2},0[. 
   \end{array}
\right.
\]
  For  $z\in [0,1]$ we define $F_{p,N}(z)$ by 
\begin{align*}
F_{p,N}(z) =& \sum_{n_{0}=0}^\infty \frac{1}{N+1+n_{0}} \sum_{n_{1}=0}^{\infty} \frac{1}{N+1+w_{1}+w_{0}}
\times \cdots \\
\times & \sum_{n_{2p-1}=0}^\infty \frac{1}{N+1+n_{2p-2}+n_{2p-1}} \\
\times& \sum_{n_{2p}=0}^\infty \frac{1}{N+1+n_{2p-1}+n_{2p}}
\frac{1}{1+\frac{1}{N}+\frac{n_{2p}}{N}-z}.
\end{align*}

 Repeating the same  idea as previously for the sums on $n_{2m-1}, \cdots, n_{0}$ we finally obtain
$$H_{p,N} (u) = \frac{1}{N} \left( \sum_{j=0}^M \left(\frac{\sin(\pi \alpha^j)}{\pi} \right)^{2p+2}
 {\overline{ \chi_{j}}}^{u}\right) 
F_{m,N} (\frac{u}{N})+R_{N,\alpha_{1}}).
$$
with $R_{N,\alpha_{1}}$ as announced previously.\\
 For all $\alpha\in ]-\frac{1}{2}, \frac{1}{2}[$ we established in \cite{RS10} the continuity of the function $F_{p,N}$ and  the 
uniform convergence in $[0,1]$ of the sequence
$\displaystyle{\sum_{p=0}^\infty 
\left(\frac{\sin(\pi \alpha)}{\pi}\right)^{2p} F_{p,N}(z)}$.
For $\alpha\in ]-\frac{1}{2}, \frac{1}{2}[$ let us denote by $F_{N,\alpha}  (z)$ the sum
$\displaystyle{ \sum_{m=0}^{+\infty} \left( \frac{\sin \pi \alpha}{\pi}\right) ^{2 m} F_{m,N}(z)}$.
The function $F_{N,\alpha}$ is defined, continuous and derivable on $[0,1[$ (see \cite{RS10} Lemma 4). Moreover for all $z\in [0,\delta ]$ , $0<\delta <1$ we have the upper bounds
$$ \frac{1}{1+\frac{1}{N} +\frac{n_{2p}}{N} -z} \le \frac{1}{1+\frac{1}{N}  -\delta }.$$ 
Hence 
$$\left( \frac{1+\frac{1}{N}-\delta } {1+\frac{1}{N} +\frac{n_{2p}}{N} -z} \right)^2 \le
 \frac{1+\frac{1}{N}-\delta }{1+\frac{1}{N} +\frac{n_{2p}}{N} -z }$$ 
and 
$$\left(\frac{1 } {1+\frac{1}{N} +\frac{n_{2p}}{N} -z} \right)^2 \le
 \frac{1} {1+\frac{1}{N}-\delta } \frac{1}{1+\frac{1}{N} +\frac{n_{2p}}{N} -z}.$$
These last inequalities and the proof of Lemma 4 in \cite{RS10} provide that
$F_{N,\alpha}$ is in $C^1[0,1[$.

Always in \cite{RS10} we have obtained that, for all $z$ in $[0,1]$,
\begin{equation}\label{F}
\Bigl \vert   F_{N,\alpha}(z) \Bigr \vert \le K_{0} \left( 1+ \Bigr \vert \ln ( 1-z+\frac{1+\alpha}{N})\Bigl \vert \right)
\end{equation}
where $K_{0}$ is a constant no depending from $N$.\\
 Now we have to prove the point ii) of the statement. For $z,z' \in [0,\delta]$ 
 \begin{align*}
 &\Bigl \vert \frac{ z-z'}{ (1+\frac{1+\alpha}{N} +
 \frac{n_{2m}}{N} -z) (1+\frac{1+\alpha}{N} +
 \frac{n_{2m}}{N} -z')}\Bigr\vert \\
& \le  \frac{\vert z-z'\vert }{ 1-\delta} 
\frac{1}{ 1+\frac{1+\alpha}{N} +
 \frac{n_{2m}}{N} -\delta}
 \end{align*}
 that implies, with (\ref{F})
\begin{equation}\label{unifcont1}
\vert F_{N,\alpha} (z) -F_{N,\alpha} (z')Ê\vert 
 \le \vert z-z'\vert \frac{ K_0 \left ( 1+ \Bigl \vert
 \ln (1-\delta +\frac{1+\alpha}{N})\Bigr \vert \right)}
 {1-\delta}.
 \end{equation}
 In the same way we have 
 \begin{align*}
 &\vert z-z'\vert \Bigl \vert \frac{ ((1+\frac{1+\alpha}{N} +
 \frac{n_{2m}}{N} -z)+ (1+\frac{1+\alpha}{N} +
 \frac{n_{2m}}{N} -z')}{ (1+\frac{1+\alpha}{N} +
 \frac{n_{2m}}{N} -z)^2 (1+\frac{1+\alpha}{N} +
 \frac{n_{2m}}{N} -z')^2}\Bigr\vert \\
& \le 2 \vert z-z'\vert 
  \frac{1}{ (1-\delta)^2}  
\frac{1}{ 1+\frac{1+\alpha}{N} +
 \frac{n_{2m}}{N} -\delta}
 \end{align*}
and 
 always with the inequality (\ref{F})
\begin{equation}\label{unifcont2}
\vert F'_{N,\alpha} (z) -F'_{N,\alpha} (z')Ê\vert 
 \le2 \vert z-z'\vert \frac{ K_0 \left ( 1+ \Bigl \vert
 \ln (1-\delta +\frac{1+\alpha}{N})\Bigr \vert \right)}
 {(1-\delta)^2}.
 \end{equation}
Using  (\ref{unifcont1}) and (\ref{unifcont2})
we get the point $ii)$.\\
 To achieve the proof we have  to remark that the uniformity in $k$ in the point $iii)$ is a direct consequence of Property 2. 
\end{proof}
We have now to state the following lemma.  
\begin{lemma}\label{final}
For $\frac{k}{N}\rightarrow x$, $0<x<1$ we have, with the notations of Theorem \ref{COEF},
$$ \sum_{u=0}^k \beta_{k-u} \left( \sum_{j=1}^M F_{\alpha_{j},N} (\frac{u}{N})
 {\overline{\chi_{j}}}^{u}\right)
= \left(\sum_{j=1}^m  {\overline{\chi_{j}}}^u  c_{1}^{-1} (\chi_{j}) K_{j}\right) 
\sum_{u=0}^k \beta_{k-u}^{(\alpha_{1})} F_{N\alpha_{1}} (\frac{u}{N})+o(k^{\alpha_{1}-1}),$$
uniformly in $k$ for $x$ in all compact of $]0,1[$ and for $K_{j}$ as in Property \ref{PROP1}.
\end{lemma}
\begin{remark}
This Lemma and Lemma \ref{INVERS3}  imply  the equality
$$
T_{N} ^{-1} \left(f \right)_{k+1,1} = \left(\sum_{j=1}^m  {\overline{\chi_{j}}}^u  c_{1}^{-1} (\chi_{j}) K_{j}\right)
T_{N} ^{-1} \left( \vert 1-\chi\vert^{2\alpha_{1}}\right)_{k+1,1} +o(k^{\alpha_{1}-1})
$$
with (see \cite{RS10} Lemma 3)
$$T_{N} ^{-1} \left( \vert 1-\chi\vert ^{2\alpha_{1}} \right)_{k+1,1}= \left( \beta_{k} ^{(\alpha)} - \frac{1}{N} \sum_{u=0}^k  \beta_{k-u}^{(\alpha_{1})} F_{N,\alpha_{1}}(\frac{u}{N})\right).
$$
\end{remark}
\begin{proof}{of lemma \ref{final}}
With our notation assume $x \in [0,\delta]$, $0<\delta <1$.
Put 
$k_0 =N^\gamma$ with $\gamma\in ]\max(\frac{\alpha_{1}}{\tau_{1}}, \frac{-\alpha_{1}}{1-\alpha_{1}}),1[$ if $\alpha_{1}<0$, and 
 $\gamma\in ]0,1[$ if $\alpha_{1}>0$. For all integer $h$, $0\le h \le M$ we can splite the sum 
$\displaystyle{
  \sum_{u=0}^k  \beta_{k-u} F_{N,\alpha_{h}} (\frac{u}{N})\overline{\chi_{h}}^u}$ into 
$S=\displaystyle{ \sum_{u=k-k_0}^{k} \beta_{k-u} F_{N,\alpha_{h}} (\frac{u}{N})\overline{\chi_{h}}^u }$
and 
$ S'=\displaystyle{\sum_{u=0}^{k-k_0}\beta_{k-u} F_{N,\alpha_{h}} (\frac{u}{N})\overline{\chi_{h}}^u}. $
First we assume that $0\le h \le m$. Then $\alpha_{1}= \alpha_{h}$ and Property \ref{PROP1} and the assumption on  $\tau_{1}$ show that 
\begin{align*}
  S' &=  
 \sum_{u=0}^{k-k_0} \left( \sum_{j=1}^m K_{j} \overline{\chi_{j}} ^{k-u} c_{1}^{-1} (\chi_{j}) \right)
 \frac{(k-u)^{\alpha_{1}-1}}{\Gamma(\alpha_{1})} F_{N,\alpha_{h}} (\frac{u}{N})\overline{\chi_{h}}^u 
\\ &
  = K_{h} \overline{\chi_{h}} ^{k} c_{1}^{-1} (\chi_{h}) \sum_{u=0}^{k-k_0}
  \frac{(k-u)^{\alpha_{1}-1}}{\Gamma(\alpha_{1})} F_{N,\alpha_{1}} (\frac{u}{N}) 
  \\
 & + \left( \sum_{j=1,j \neq h}^m K_{j} \overline{\chi_{j}}^{k} c_{1}^{-1} (\chi_{j}) \right)
  \sum_{u=0}^{k-k_0} 
\frac{(k-u)^{\alpha_{1}-1}}{\Gamma(\alpha_{1})} 
F_{N,\alpha_{1}} (\frac{u}{N})
(\overline{\chi_{h}}\chi_{j})^u)+o(k^{\alpha_{1}})
\end{align*}
uniformly in $k$ (\ref{ZYG}).
Then an Abel summation provides that the quantity\\
$ \Bigr\vert\displaystyle{\sum_{u=0}^{k-k_0} (k-u)^{\alpha-1} F_{N,\alpha_{1}} (\frac{u}{N})}
 ( \overline{\chi_{h}}\chi_{j})^u\Bigl \vert  $
is bounded by \\
$ M_1 k_{0}^{\alpha-1}+\displaystyle{\sum_{u=0}^{k-k_0} 
\Bigl\vert (k-u-1)^{\alpha-1} F_{N,\alpha_{1}} (\frac{u+1}{N})
- (k-u)^{\alpha-1} F_{N,\alpha_{1}} (\frac{u}{N})\Bigr \vert}$
with $M_1$ no depending from $k$. 
Moreover 
\begin{align*} 
&\sum_{u=0}^{k-k_0} 
\Bigl\vert (k-u-1)^{\alpha_{1}-1} F_{N,\alpha_{1}} (\frac{u+1}{N})
- (k-u)^{\alpha_{1}-1} F_{N,\alpha_{1}} (\frac{u}{N})\Bigr\vert \\
&\le 
\sum_{u=0}^{k -k_0}
\vert (k-u-1)^{\alpha_{1}-1} - (k-u)^{\alpha-1} \vert 
\vert F_{N,\alpha_{1}}(\frac{u}{N})\vert\\
&+\sum_{u=0}^{k-k_0}
\vert  F_{N,\alpha_{1}} (\frac{u+1}{N})
-  F_{N,\alpha_{1}}(\frac{u}{N})\vert  \vert(k-u-1)^{\alpha_{1}-1}\vert 
\end{align*}
From the inequality (\ref{F}) (we have assumed $0<\frac{k}{N}<\delta$) we infer
$$ \sum_{u=0}^{k-k_0} 
\vert (k-u-1)^{\alpha_{1}-1} - (k-u)^{\alpha_{1}-1} \vert 
\vert F_{N,\alpha_{1}} (u)\vert\le M_2\sum_{w=k_{0}}^{k}
 v^{\alpha_{1}-2}$$
 with $M_{2}$ no depending from $k$. We finally get 
\begin{align*}
\sum_{u=0}^{k-k_0} 
\vert (k-u-1)^{\alpha_{1}-1} - (k-u)^{\alpha_{1}-1} \vert 
\vert F_{N,\alpha_{1}} (u)\vert &= O\left( \sum_{w=k_{0}}^{k}
 v^{\alpha_{1}-2}\right) \\
&= O \left( k_0 ^{\alpha_{1}-1}\right) =o(k^{\alpha_{1}})
 \end{align*} 
 Identically  Lemma \ref{INVERS3} and the main value 
 theorem provide (since $F_{N,\alpha}\in C^1[0,\delta ], \quad \forall \delta \in ]0,1[$).
 $$ 
 \sum_{u=0}^{k-k_0}
\vert  F_{N,\alpha_{1}} (\frac{u+1}{N})
-  F_{N,\alpha_{1}}(\frac{u}{N})\vert  \vert(k-u-1)^{\alpha_{1}-1}\vert \le M_3 \frac{k^\alpha_{1}}{N} =
o(k^{\alpha_{1}}).$$
  always with $M_3$ no depending from $N$.
By definition of $k_0$ and with Property \ref{PROP1} we have easily the existence of a constant $M_{4}$, always no depending from $k$,
satisfying for $\alpha_{1}>0$ 
$ \displaystyle{ \sum_{u=k-k_{0}}^k \beta_{k-u}^{(\alpha_{1})} F_{N,\alpha_{1}} (\frac{u}{N}) \le 
 M_4 k_{0}^{\alpha_{1}} =o(k^{\alpha_{1}})}$.
 Consequently for $\alpha_{1}>0$ and $0\le h \le m$ 
\begin{align*}
 & \sum_{u=0}^k  \beta_{k-u} F_{N,\alpha_{h}} (\frac{u}{N})\overline{\chi_{h}}^u
 \\
 & = K_{h }\overline{\chi_{h}} ^{k} c_{1}^{-1} (\chi_{h}) \sum_{u=0}^{k-k_0}
  \frac{(k-u)^{\alpha_{1}-1}}{\Gamma(\alpha_{1})} F_{N,\alpha_{1}} (\frac{u}{N}) +o(k^{\alpha-1})\\
  = K_{h }\overline{\chi_{h}} ^{k} c_{1}^{-1} (\chi_{h}) \sum_{u=0}^{k-}
  \beta_{k-u}^{(\alpha_{1})} F_{N,\alpha_{1}} (\frac{u}{N})o(k^{\alpha_{1}})
\end{align*}
 uniformly in $k$ with the definition of the constants
 $M_i$, $1\le i \le 4$. 
 For $h>m$ we obtain identically that 
  $$
   \sum_{u=0}^{k-k_{0}}  \beta_{k-u} F_{N,\alpha_{h}} (\frac{u}{N})\overline{\chi_{h}}^u =
   o(k^{\alpha_{1}}).
   $$
and we get the Lemma for $\alpha_{1}>0$.\\
Hence we assume in the rest of 
the demonstration that  $\alpha_{1} \in ]-\frac{1}{2},0[.$ Recall that now 
$\gamma\in ]\max (\frac{\alpha_{1}}{\beta},
\frac{-\alpha_{1}}{1-\alpha_{1}}),1[$.\\
We have to evaluate the sum 
 $\displaystyle{\sum_{u=k-k_0}^{k} \beta_{k-u} F_{N,\alpha_{h}} (\frac{u}{N})\overline{\chi_{h}}^u }$.
 $F_{N,\alpha_{h}} \in C^1[0, \delta ]$ implies, for $\frac{k-k_{0}}{N}\le \frac{u}{N} \le \frac{k}{N} 
\le \delta < 1,$ 
$$ F_{N,\alpha_{h}} (\frac{u}{N})  - F_{N,\alpha_{h}} (\frac{k}{N}) + F_{N,\alpha_{h}} (\frac{k}{N}) =
 F_{N,\alpha_{h}} (\frac{k}{N})  + O(\frac{k_{0}}{N}) = 
 F_{N,\alpha_{h}} (\frac{k}{N})+o(k^{\alpha_{1}})$$
  uniformly in $k$ (see once a more the definition of $\gamma$ and $\tau_{1}$).\\
Hence we can write, uniformly in $k$, 
\begin{align*}
 \sum_{u=k-k_0}^{k} \beta_{k-u} F_{N,\alpha_{h}} (\frac{u}{N})\overline{\chi_{h}}^u 
&= 
\overline{\chi_{h}}^{k} \sum_{u=k-k_0}^{k} \beta_{k-u} F_{N,\alpha_{h}} (\frac{k}{N})\chi_{h}^{k-u}
+o(k^{\alpha_{1}})
\\
&
= - \overline{\chi_{h}}^{k}F_{N,\alpha_{h}} (\frac{k}{N}) \sum_{v=k_{0}+1}^{+\infty} 
\beta_{v} \chi_{h}^{v}+o(k^{\alpha_{1}}).
\end{align*}
If $0\le h \le m$ we get 
$$ \sum_{v=k_{0}+1}^{+\infty} 
\beta_{v} \chi_{h}^{v} = \sum_{v=k_{0}+1}^{+\infty}
 \left(\sum_{j=1} ^m K_{j} c_{1}^{-1} (\chi_{j}) \overline{\chi_{j}}^v\right)
 \frac{v^{\alpha_{1}-1}}{\Gamma(\alpha_{1})} \chi_{h}^{v} +o(k_{0}^{\tau_{1}}),
$$
that is also, with the definition  $k_{0} =k^\gamma$, $\gamma\in]\max (\frac{\alpha_{1}}{\tau_{1}},
\frac{-\alpha_{1}}{1-\alpha_{1}}),1[$, 
$$ \sum_{v=k_{0}+1}^{+\infty} 
\beta_{v} \chi_{h}^{v} = \sum_{v=k_{0}+1}^{+\infty} \left(  \sum_{j=1} ^m K_{j} c_{1}^{-1} (\chi_{j})\overline{\chi_{j}}^v\right)
\frac{v^{\alpha_{1}-1}}{\Gamma(\alpha_{1})} \chi_{h}^{v} +o(k^{\alpha_{1}}).
$$
We have 
$$ \sum_{v=k_{0}+1}^{+\infty} \left(  \sum_{j=1} ^m K_{j} c_{1}^{-1} (\chi_{j})\overline{\chi_{j}}^v\right)
 \frac{v^{\alpha_{1}-1}}{\Gamma(\alpha_{1})} \chi_{h}^{v}
= K_{h} c_{1}^{-1} (\chi_{h})\sum_{v=k_{0}+1}^{+\infty}\frac{v^{\alpha_{1}-1}}{\Gamma(\alpha_{1})} 
+R$$
 
 An Abdel summation provides $\vert R\vert \le M_{5} k_{0}^{\alpha_{1}-1} =o(k^{\alpha_{1}})$ uniformly in 
 $k$.
 
Hence we have 
$$ \sum_{u=k-k_0}^{k} \beta_{k-u} F_{N,\alpha_{h}} (\frac{u}{N})\overline{\chi_{h}}^u  
= -K_{h} c_{1}^{-1} (\chi_{h}) \overline{\chi_{h}}^k F_{\alpha_{1}}  (\frac{k}{N}) \sum_{v=k_{0}+1}^{+\infty}\frac{v^{\alpha_{1}-1}}{\Gamma(\alpha_{1})} 
+ o(k^{\alpha_{1}})$$
that is also
\begin{align*}
 \sum_{u=k-k_0}^{k} \beta_{k-u} F_{N,\alpha_{h}} (\frac{u}{N})\overline{\chi_{h}}^u  
&= K_{h} c_{1}^{-1} (\chi_{h}) \overline{\chi_{h}}^k F_{\alpha_{1}}  (\frac{k}{N} )\sum_{u=k-k_{0}}^{k}
\frac{\beta_{k-u}^{(\alpha_{1})}}{\Gamma(\alpha_{1})} 
+ o(k^{\alpha_{1}})\\
& = K_{h} c_{1}^{-1} (\chi_{h}) \overline{\chi_{h}}^k \sum_{u=k-k_{0}}^{k}
\frac{\beta_{k-u}^{(\alpha_{1})}}{\Gamma(\alpha_{1})} 
F_{\alpha_{1}}  (\frac{u}{N} )+ o(k^{\alpha_{1}})
\end{align*}
uniformly in $k$.
Since we have seen that the sum
$$
\sum_{u=0}^{k-k_0}\beta_{k-u} F_{N,\alpha_{h}} (\frac{u}{N})\overline{\chi_{h}}^u
$$
is equal to 
$$\overline{\chi_{h}} ^{k} c_{1}^{-1} (\chi_{h}) \sum_{u=0}^{k-k_0}
  \frac{\beta_{k-u}^{(\alpha_{1})}}{\Gamma(\alpha_{1})} F_{N,\alpha_{1}} (\frac{u}{N}) 
+ o(k^{\alpha_{1}})  $$
 we can also conclude, as for $\alpha_{1}>0$, that for $1\le h \le m$
$$
  \sum_{u=0}^{k}  \beta_{k} F_{N,\alpha_{h}}( \frac{u}{N}) \overline{\chi_{h}}^u
=K_{h} c_{1}^{-1} (\chi_{h}) \overline{\chi_{h}}^k \sum_{v=0}^{k}
\frac{\beta_{k-u}^{(\alpha_{1})}}{\Gamma(\alpha_{1})} 
F_{\alpha_{1}}  (\frac{v}{N} )+ o(k^{\alpha_{1}}).
$$
  Identically if $h>m$ we obtain 
 $\Bigl \vert \displaystyle{\sum_{u=0}^{k} \beta_{k-u} F_{N,\alpha_{h}} (\frac{u}{N})\overline{\chi_{h}}^u 
\Bigr \vert =  o(k^{\alpha_{1}})}$ uniformly in $k$.
 The uniformity is clearly provided by the uniformity in Lemma \ref{INVERS3} and by the previous 
 remarks.
 This last remark is sufficient to prove Lemma \ref{final}.
\end{proof}
Then Theorem \ref {COEF} is a direct consequence of the inversion formula and of Lemma 
\ref{final}.\section{Appendix}
\subsection{Estimation of a trigonometric sum}
\begin{lemma}\label{APPENDIX1}
Let $M_{0},M_{1}$ two integers with $0<M_{0}<M_{1}$, $\chi\neq 1$  and $f$ a function in 
$\mathcal C^1\left(]M_{0},M_{1}[\right)$  with for all $t\in ]M_{0},M_{1}[$ $f(t) = O(t^{\beta})$ 
and $f'(t) = O(t^{\beta-1})$. Then 
$$\Bigl \vert\sum_{u=M_{0}}^{M_{1}} f(u) \chi^u  \Bigr \vert = \Bigl \lbrace \begin{array}{l}
 O(M_{1}^\beta) \quad \mathrm{if} \quad \beta>0\\
O(M_{0}^\beta) \quad \mathrm{if} \quad \beta<0.
\end{array}$$
\end{lemma}
\begin{proof}{}
With an Abel summation we obtain, if $\sigma_{u}= 1 + \cdots + \chi^u$,  
$$
\sum_{u=M_{0}}^{M_{1}} f(u) \chi^u  = \sum_{u=M_{0}}^{M_{1}-1} \left( f(u+1) -f(u)\right) 
\sigma_{u} +f(M_{1}) \sigma_{M_{1}} + f(M_{0}) \sigma_{M_{0}-1} 
$$
and
\begin{eqnarray*}
\sum_{u=M_{0}}^{M_{1}-1} \left( f(u+1) -f(u)\right) 
\sigma_{u} 
&=& \left(f(M_{0}) +f(M_{1}) \right)\left( \frac{1}{1-\chi}\right) - 
\sum_{u=M_{0}}^{M_{1}-1} \left( f(u+1) -f(u) \right) \frac{\chi^{u+1}}{1-\chi}\\
& =& \sum_{u=M_{0}}^{M_{1}-1} f'(c_{u}) \frac{\chi^{u+1}}{1-\chi}+ \left(f(M_{0}) +f(M_{1}) \right)\left( \frac{1}{1-\chi}\right)
\end{eqnarray*}
with $c_{u}\in ]u,u+1[$. 
We have 
$$\Bigl \vert   \sum_{u=M_{0}}^{M_{1}-1} f'(c_{u}) \frac{\chi^{u}}{1-\chi}\Bigr\vert \le 
O\left( \sum_{u=M_{0}}^{M_{1}-1} u^{\beta-1} \right)$$
hence 
$$\Bigl \vert\sum_{u=M_{0}}^{M_{1}} f(u) \chi^u  \Bigr \vert = \Bigl \lbrace
 \begin{array}{l}
 O(M_{1}^\beta) \quad \mathrm{if} \quad \beta>0\\
O(M_{0}^\beta) \quad \mathrm{if} \quad \beta<0.
\end{array}$$
\end{proof}

 \bibliography{Toeplitzdeux}

\begin{thebibliography}{10}

\bibitem{BAS1}
E.~L. Basor.
\newblock {T}oeplitz determinants,{F}isher-{H}artwig symbols and random
  matrices.
\newblock In lecture note~series London~mathematical society, editor, {\em
  Recent perspectives in random matrix theory and number theory}, volume 322,
  pages 309--336, 2005.

\bibitem{RamBeau}
P.~Beaumont and R.~Ramachandran.
\newblock Robust estimation of {G}{A}{R}{M}{A} model parameters with an
  application to cointegration among interest rates of industrialized country.
\newblock {\em Computatinal economics}, 17:179--201, 2001.

\bibitem{B}
J.~Beran.
\newblock {\em Statistics for long memory process}.
\newblock Chapmann and Hall, 1994.

\bibitem{Bo.4}
A.~B\"ottcher and B.~Silbermann.
\newblock Toeplitz matrices and determinants with {F}isher-{H}artwig symbols.
\newblock {\em J. Funct. Anal.}, 63:178--214, 1985.

\bibitem{Bo.5}
A.~B\"ottcher and B.~Silbermann.
\newblock Toeplitz operators and determinants generated by symbols with one
  {F}isher-{H}artwig singularity.
\newblock {\em Math. Nachr.}, 127:95--124, 1986.

\bibitem{BrDa}
P.~J. Brockwell and R.~A. Davis.
\newblock {\em Times series: theory and methods}.
\newblock Springer Verlag, 1986.

\bibitem{WWCG}
Q.C. Cheng, H.~L. Gray, and W.~A. Wayne.
\newblock A k-factor {G}{A}{R}{M}{A} long-memory model.
\newblock {\em Journal of time series analysis}, 19(4):485--504, 1998.

\bibitem{Dahlhaus}
R.~Dahlhaus.
\newblock Efficient parameter estimation for self-similar processes.
\newblock {\em Ann. Statist.}, 17:1749--1766, 1989.

\bibitem{D01}
P.~A. Deift, K.~T.~R. McLaughlin, T.~Kriecherbauer, S~Venakides, and X.~Zhou.
\newblock A riemann-{H}ilbert approach to asymptotic questions for orthogonal
  polynomials.
\newblock {\em J. Approx. Theory}, 95:388--475, 1998.

\bibitem{De98}
P.A. Deift.
\newblock {\em Orthogonal polynomials and random matrices: a
  {R}iemann-{H}ilbert approach}.
\newblock AMS, New York, 1998.

\bibitem{AKA4}
Abdou~K\^{a} Diongue and D.~Gu\'eguan.
\newblock Estimating parameters for k-factor {G}{I}{G}{A}{R}{C}{H} process.
\newblock {\em C.R.A.S, Serie I}, 339,:435,440, 2004.

\bibitem{AKA3}
Abdou~K\^{a} Diongue and D.~Gu\'eguan.
\newblock Estimation of k-factor {G}{I}{G}{A}{R}{C}{H} process : a {M}onte
  {C}arlo {S}tudy.
\newblock {\em Communictions in {S}tatistic-{S}imulations and {c}omputations},
  37:2037,2049, 2009.

\bibitem{DOT}
P.~Doukhan, G.~Oppenheim, and M.~S. Taqqu.
\newblock {\em Theory and applications of long-range dependence}, volume~54.
\newblock Birkhuser, Boston, 2003.

\bibitem{GS}
U.~Grenander and G.~Szeg{\"o}.
\newblock {\em {T}oeplitz forms and their applications}.
\newblock Chelsea, New York, 2nd ed. edition, 1984.

\bibitem{JOHA1}
Kurt Johansson.
\newblock On random matrices from the compact classical groups.
\newblock {\em Annals of {M}athematics.}, 145:519--545, 1997.

\bibitem{KIRTEY}
A.P. Kirman and G.~Teyssiere.
\newblock {\em Long memory in economic}.
\newblock {M}athematical {R}eview, 2007.

\bibitem{Ld}
H.J. Landau.
\newblock Maximum entropy and the moment problem.
\newblock {\em Bulletin ({N}ew {S}eries) of the american mathematical society},
  16(1):47--77, 1987.

\bibitem{YILUHU}
Y.~Lu and C.~M. Hurvich.
\newblock On the complexity of the preconditioned conjugate gradient algorithm
  for solving {T}oeplitz systems with a {F}isher-{H}artwig singularity.
\newblock {\em SIAM J. Matrix Anal. Appl.}, 27:638--653, 2005.

\bibitem{ML3}
A.~Martinez-Finkelshtein, K.~T.~R McLaughlin, and E.~B. Saff.
\newblock Asymptotics of orthogonal polynomials with respect to an analytic
  weight with algebraic singularities on the circle.
\newblock {\em Internat. Math. Research Notices}, 2:423--434, 1953.

\bibitem{Nagao3}
Taro Nagao.
\newblock Universal {C}orrelations {N}ear a {S}ingularity of {R}andom {M}atrix
  {S}pectrum.
\newblock {\em Journal of the {P}hysical {S}ociety of {J}apan.}, 64:3675--3681,
  1995.

\bibitem{Nagao2}
Taro Nagao and Miki Wadati.
\newblock An {I}ntegration {M}ethodon {G}eneralized {C}ircular {E}nsembles.
\newblock {\em Journal of the {P}hysical {S}ociety of {J}apan.}, 61:1903--1909,
  1992.

\bibitem{Nagao}
Taro Nagao and Miki Wadati.
\newblock Eigenvalue distribution of random matrices at the spectrum edge.
\newblock {\em Journal of the {P}hysical {S}ociety of {J}apan.}, 62:3845--3856,
  1993.

\bibitem{RS1111}
P.~Rambour and A.Seghier.
\newblock Inversion des matrices de {T}oeplitz dont le symbole admet un z\'ero
  d'ordre rationnel positif,valeur propre minimale.
\newblock {\em Annales de la Facult\'e des Sciences de Toulouse}, XXI, n¡
  1:173--2011, 2012.

\bibitem{RS04}
P.~Rambour and A.~Seghier.
\newblock Formulas for the inverses of {T}oeplitz matrices with polynomially
  singular symbols.
\newblock {\em Integr. equ. oper. theory}, 50:83--114, 2004.

\bibitem{RS09}
P.~Rambour and A.~Seghier.
\newblock Th\'eor\`emes de trace de type {S}zeg\"o dans le cas singulier.
\newblock {\em Bull. des {S}ci. Math.}, 129:149--174, 2005.

\bibitem{RS10}
P.~Rambour and A.~{S}eghier.
\newblock Inverse asymptotique des matrices de {T}oeplitz de symbole $(1-\cos
  \theta)^\alpha f_{1},$ $ \frac{-1}{2}<\alpha\le \frac{1}{2}$, et noyaux
  int\'egraux.
\newblock {\em Bull. des {S}ci. {M}ath.}, 134:155--188, 2008.

\bibitem{BS1}
B.~Simon.
\newblock {\em Orthogonal polynomials on the unit circle, {P}art 1: classical
  theory}, volume~54.
\newblock {A}merican {M}athematical {S}ociety, 2005.

\bibitem{BS2}
B.~Simon.
\newblock {\em Orthogonal polynomials on the unit circle, {P}art 2: spectral
  theory}, volume~54.
\newblock {A}merican {M}athematical {S}ociety, 2005.

\bibitem{SZEG}
G.~Szeg\"o.
\newblock {\em Orthogonal polynomials}.
\newblock American {M}athematical {S}ociety, colloquiu\`m publication,
  {P}rovidence, {R}hodes {I}sland, 3nd edition, 1967.

\bibitem{TW01}
C.A. Tracy and H.~Widom.
\newblock Correlation functions, cluster functions and spacing distribution for
  random matrices.
\newblock {\em J. {S}tat. {P}hys.}, 92:809--835, 1999.

\bibitem{Zyg2}
A.~Zygmund.
\newblock {\em Trigonometric series}, volume~1.
\newblock Cambridge University Press,, 1968.

\end{thebibliography}

\end{document}